\documentclass[12pt,reqno]{amsart}
\usepackage{amsmath} 
\usepackage{amssymb}
\usepackage{dsfont}
\usepackage[dvips,draft,final]{graphics}
\usepackage[T1]{fontenc}
\usepackage{fancyhdr}
\usepackage{url}
\usepackage[colorlinks,linktocpage,linkcolor=blue]{hyperref}

\usepackage{color}
\usepackage{graphicx}
\newtheorem{thm}{Theorem}[section]

\newtheorem{lem}{Lemma}[section]
\newtheorem{prop}{Proposition}[section]

\renewenvironment{abstract}{%
        \small
        \quotation
         \noindent {\bfseries \abstractname } }%
      {\if@twocolumn\else\endquotation\fi}
\numberwithin{equation}{section}

\renewcommand{\leq}{\leqslant}
\renewcommand{\geq}{\geqslant}
\providecommand{\abs}[1]{\left\lvert#1\right\rvert}
\providecommand{\norm}[1]{\left\lVert#1\right\rVert}
\newcommand{\bel}{\begin{equation} \label}
\newcommand{\ee}{\end{equation}}
\def\R{\mathbb R}

\def\pa{\partial}

\renewcommand{\leq}{\leqslant}
\renewcommand{\geq}{\geqslant}
\def\epsilon{\varepsilon}
\def\phi {\varphi}

\title[Simultaneous recovery of various time-dependent coefficients]{Simultaneous recovery of various time-dependent coefficients in a hyperbolic equation  with space-time nonlocal attenuation
}

\author[Yavar Kian and Gunther Uhlmann]{ Yavar Kian$^1$ and Gunther Uhlmann$^2$
}
\date{}

\begin{document}

\maketitle

\begin{abstract}  This article is devoted to the inverse problem of uniquely determining different classes of coefficients in a hyperbolic equation with space-time nonlocal attenuation from a local source-to-solution map. Our objective is to exploit the nonlocal attenuation mechanism to address inverse coefficient problems that remain open, or are even fundamentally intractable, for linear and nonlinear hyperbolic equations without attenuation. More precisely, by leveraging the interactions induced by nonlocal attenuation, we simultaneously determine general classes of leading and lower order coefficients depending on both space and time. In addition, we formulate our results using data collected on disjoint sets, thereby extending the existing theory and introducing the notion of data on sets disjoint in both space and time for hyperbolic equations. Our approach builds on key features of nonlocal operators, such as memory effects and the unique continuation principle, while integrating the interplay between local and nonlocal operators along with structural properties of hyperbolic equations.\\
{\bf Mathematics subject classification 2020 :} 35R30, 35L20, 26A33.
\vskip 4.5mm

\end{abstract}

\renewcommand{\thefootnote}{\fnsymbol{footnote}}
\footnotetext{\hspace*{-5mm} 
\begin{tabular}{@{}r@{}p{16cm}@{}}
$^*$
&The work of Y. Kian is supported by the French National Research Agency ANR and Hong Kong RGC Joint Research Scheme for the project IdiAnoDiff (grant ANR-24-CE40-7039). The work of G. Uhlmann is partially supported by NSF.\\
$^1$ 
& Univ Rouen Normandie, CNRS, Normandie Univ, LMRS UMR 6085, F-76000 Rouen, France  (\texttt{yavar.kian@univ-rouen.fr})\\
$^2$
& Department of Mathematics, University of Washington , Seattle, WA 98195-4350, USA. (\texttt{gunther@math.washington.edu})
\end{tabular}}

\section{Introduction}
\label{sec-intro}

Let $n\geq2$, $T>0$,  $\alpha\in(0,1)$, $s\in(0,1/2]$,  $c\in C^2([0,T]\times\R^n;(0,+\infty))$, $a=(a_{ij})_{1\leq i,j\leq n}\in C^1([0,T]\times\R^n;\R^{n\times n})$, $B=(b_k)_{0\leq k\leq n}\in C^1([0,T]\times\R^n;\R^{n+1})$, $d,e\in C^1([0,T]\times\R^n)$. We assume that there exists $R_1>R_0>0$   such that
\bel{coe1} d(t,x)=e(t,x)=0,\quad (t,x)\in[0,T]\times\R^{n},\ |x|>R_1,\ee
\bel{coe2} c(t,x)=1,\quad a_{ij}(t,x)=\delta_{ij},\quad B(t,x)=0,\quad (t,x)\in[0,T]\times\R^{n},\ |x|>R_0,\ i,j=1,\ldots,n,\ee
where $\delta_{ij}$, $i,j=1,\ldots,n$, denotes the Kronecker symbol. Moreover, we assume that there exists a constant $c_\star>0$ such that
\bel{coe3} a_{ij}(t,x)=a_{ji}(t,x),\quad (t,x)\in[0,T]\times\R^{n},\ i,j=1,\ldots,n,\ee
\bel{coe4} \sum_{i,j=1}^na_{ij}(t,x)\xi_i\xi_j\geq c_\star|\xi|^2,\quad (t,x)\in[0,T]\times\R^{n},\ \xi=(\xi_1,\ldots,\xi_n)\in\R^n.\ee
We fix also $P_{a,B,c}$ the differential operator defined by 
\bel{P}P_{a,B,c}w=c(t,x)^{-2}\partial_t^2 w-\sum_{i,j=1}^n\partial_{x_i}(a_{ij}(t,x)\partial_{x_j}w)+B(t,x)\cdot(\partial_t w,\nabla_xw),\quad w\in H^1((0,T)\times\R^n)\ee
and, for any $F\in L^2((0,T)\times\R^n)$, we consider the following initial value problem (IVP in short)
\begin{equation}
\label{eq1}
\left\{
\begin{aligned}
& P_{a,B,c}u +d(t,x) \partial_t^{\alpha} u +e(t,x)(-\Delta)^{s}u = F, &&\quad \textrm{in } (0,T)\times \R^n,\\                                                    
& u(0,x) = 0,\quad \partial_t u(0,x) = 0,                                             &&\quad x\in\R^n.
\end{aligned}
\right.
\end{equation}
Here $\partial_t^\alpha$ denotes
the Caputo fractional derivative of order $\alpha $ with respect to $t$, defined by 
$$
\partial_t^{\alpha} u(t) := 
\frac{1}{\Gamma(1-\alpha)} \int_0^t (t-s)^{-\alpha}\partial_s u(s) ds,\ u\in W^{1,1}(0,T),\ t\in(0,T),
$$
where $\Gamma$ denotes the usual Gamma function. In addition,  for all $r>0$, we denote by $(-\Delta)^r$  the fractional Laplacian of order $r$ defined by
$$(-\Delta)^rh=\mathcal F^{-1}_\xi(|\xi|^{2r}\mathcal F h(\xi)),\quad h\in L^2(\R^n),$$
where $\mathcal F$ denotes the Fourier transform in $\R^n$.

Fix $\epsilon\in(0,T)$ and $\mathcal O_j$, $j=1,2$, two open not empty subsets of $\R^n$ such that $\overline{\mathcal O_1}\cap\overline{\mathcal O_2}=\emptyset$ and $\overline{\mathcal O_j}\subset B_{R_1}\setminus\overline{B_{R_0}}$, with $B_r:=\{x\in\R^n:\ |x|<r\}$, $r>0$. Applying Proposition \ref{p1}, we define the partial source-to-solution map $\Lambda_{a,B,c,d,\epsilon}$ as follows
\bel{ss} \Lambda_{a,B,c,d,\epsilon}: C^\infty_0((0,\epsilon)\times\mathcal O_1)\ni F\mapsto u|_{(T-\epsilon,T)\times\mathcal O_2},\ee
with $u\in C([0,T];H^1(\R^n))\cap C^1([0,T];L^2(\R^n))$ the unique solution of \eqref{eq1}. In the present article, we study the following inverse problem:

\begin{itemize}
\item[{\bf(IP)}] {\em
For $\epsilon\in(0,T)$ arbitrarily small, determine the set of coefficients $(a,B,c,d)$ from the knowledge of $\Lambda_{a,B,c,d,\epsilon}$.
}
\end{itemize}

Recall that hyperbolic equations of the form \eqref{eq1} incorporating nonlocal attenuation terms are commonly used to model various physical phenomena. These include imaging modalities such as tomography in biological tissues, which are often modeled by acoustic equations with frequency-dependent attenuation terms \cite{HP,Sz}. Such attenuation can be expressed in terms of both time-fractional derivatives and the fractional Laplacian; see, for instance, \cite{BC,CH}. Similarly, equations of the form \eqref{eq1} arise in the modeling of wave propagation in viscoelastic media \cite{AEE,SZS}, including more specific models such as the space-time fractional Zener wave equation, which involves both time-fractional derivatives and the fractional Laplacian \cite{AJOPZ}. In these models, the time-fractional attenuation $d\partial_t^{\alpha}$ accounts for memory and hereditary effects in the material, whereas the attenuation term $e(-\Delta)^{s}$, involving fractional Laplacians, accounts for long-range spatial interactions. In this context, the goal of the inverse problem \textbf{(IP)} is to determine the physical properties of the medium (e.g., density, attenuation, wave speed) encoded by the general time-dependent coefficients $(a,B,c,d)$ in \eqref{eq1}.

Inverse coefficient problems for hyperbolic equations, such as \textbf{(IP)}, represent a central theme in inverse problem theory and have been extensively investigated over the past several decades. For time-independent coefficients, one of the most effective approaches to solving \textbf{(IP)} is the \emph{boundary control method}, introduced in \cite{B,BK} (see also \cite{KKL}). However, limitations of unique continuation in the presence of time-dependent coefficients \cite{AB} have so far restricted the application of the boundary control method to coefficients that depend analytically on time \cite{Es}. To overcome these limitations, several authors have developed approaches based on the construction of special solutions, commonly referred to as geometric optics solutions. Without aiming to provide an exhaustive account of the literature, we mention \cite{FIKO,FeK,Ki1,KiOk,SY}, as well as the more recent works \cite{AFO1,AFO2}, which address related inverse problems in the setting of Lorentzian manifolds. For hyperbolic equations involving nonlocal terms, we mention results concerning source recovery \cite{AP,AP2,HKSY}, as well as the identification of time-independent coefficients \cite{BuID,BDU,Dy,Y,Z}. More recently, \cite{KRU} investigated the recovery of zero- and first-order time-dependent coefficients in hyperbolic equations with time-fractional attenuation on a Riemannian manifold. Their approach exploits the memory effects induced by the fractional attenuation term, allowing the boundary measurements to be restricted to an arbitrarily small time interval of the form $(T-\epsilon,T)$, with $\epsilon>0$ arbitrarily small.

The aforementioned results have primarily focused on the recovery of lower-order time-dependent coefficients, such as $B$, while leaving aside the recovery of the leading-order coefficients $a$ and $c$. The recovery of such coefficients from the source-to-solution map of a hyperbolic equation remains an open problem in the linear setting. In view of these limitations, the work \cite{KLU} addressed the recovery of time-dependent leading-order coefficients for a semilinear hyperbolic equation. The general strategy developed in \cite{KLU} relies on nonlinear interactions and on the properties of products of waves, thereby allowing the authors to solve, in the nonlinear setting, an inverse problem that remains open for the corresponding linear equation. The approach initiated in \cite{KLU} has subsequently been successfully developed in several directions \cite{CLOP,FLO,HUZ,KLOU}. In the spirit of the methodology introduced in \cite{KLU}, the goal of the present article is to develop a new approach to inverse problems for hyperbolic equations with nonlocal attenuation, in which the nonlocal character of the attenuation is substantially leveraged to address the inverse problem \textbf{(IP)}.

In addition, owing to the generality of the class of coefficients under consideration, problem \textbf{(IP)} also falls within the category of inverse problems involving data prescribed on disjoint sets. More precisely, problem \textbf{(IP)} introduces the notion of coefficient recovery from measurements supported on sets that are disjoint both in time and in space. Indeed, the formulation of \textbf{(IP)} involves excitations supported in $(0,\epsilon)\times\mathcal O_1$ and measurements performed on $(T-\epsilon,T)\times\mathcal O_2$, where, for $\epsilon\in(0,T/4)$, we have
$$
[0,\epsilon]\cap[T-\epsilon,T]=\emptyset,
\qquad
\overline{\mathcal O_1}\cap\overline{\mathcal O_2}=\emptyset.
$$
To the best of our knowledge, this setting extends the existing literature on inverse problems with data supported on disjoint sets, encompassing both linear and nonlinear PDEs, even in the case of time-independent coefficients \cite{FLO,KKLO,KRU,LO}.

This article is organized as follows. In Section~\ref{s2}, we present our main result, formulated in Theorem~\ref{t1}, which provides a positive resolution to problem~\textbf{(IP)} under suitable assumptions. In Section~\ref{s3}, we collect several preliminary results, including the well-posedness of problem~\eqref{eq1} and its formal adjoint, along with key properties of the nonlocal operators appearing in \eqref{eq1}. Finally, Section~\ref{s4} is devoted to the proof of Theorem~\ref{t1}.

\section{Main result}\label{s2}
Let us fix $\gamma\in C([0,T];\mathcal O_1)$, $\tau\in C([0,T];(0,+\infty)$ such that 
$$B(\gamma(t),\tau(t)):=\{x\in\R^n:\ |x-\gamma(t)|<\tau(t)\}\subset \mathcal O_2,\quad t\in[0,T].$$
Fix $\delta\in(0,T)$ and, for all $t_1\in(0,T)$, introduce the set
\bel{Q}Q_{\gamma,\tau,t_1}:=\{(t,x)\in(T-t_1,T)\times\R^n:\ x\in B(\gamma(t),\tau(t))\}\subset (T-t_1,T)\times\mathcal O_2.\ee
We fix also $U$, $V_1,V_2$ three open  not-empty subsets of $\R^n$ such that 
\bel{U}([T-\delta,T]\times\overline{U})\cap\overline{Q_{\gamma,\tau,\delta}}=\emptyset,\quad \overline{U}\subset \R^n\setminus\overline{B_{R_0}},\quad\overline{V_1}\cup\overline{V_2}\subset \mathcal O_1,\quad \overline{V_1}\cap\overline{V_2}=\emptyset.\ee

Our main result, is the following answer to the problem \textbf{(IP)}.

\begin{thm}\label{t1} For $\ell=1,2$, let $c^\ell\in C^2([0,T]\times\R^n;(0,+\infty))$, $a^\ell=(a_{ij}^\ell)_{1\leq i,j\leq n}\in C^1([0,T]\times\R^n;\R^{n\times n})$, $B^\ell=(b_k^\ell)_{0\leq k\leq n}\in C^1([0,T]\times\R^n; \R^{n+1})$, $d^\ell,e\in C^1([0,T]\times\R^n)$ satisfy conditions \eqref{coe1}-\eqref{coe4} with $(a,B,c,d)=(a^\ell,B^\ell,c^\ell,d^\ell)$.
Assume that conditions  \eqref{U}   and 
\bel{t1b}\inf_{(t,x)\in \overline{Q_{\gamma,\tau,\delta}}\cup ([0,T]\times \overline{U})\cup ([0,T]\times \overline{B_{R_0}})}|e(t,x)|>0\ee
hold true. Assume also that the following conditions 
\bel{t1c}\min\left( \inf_{x\in \overline{U}} |d^1(T,x)|,\inf_{(t,x)\in [0,T]\times \overline{V_2}} |d^1(t,x)|\right)>0\ee
\bel{t1d} d^1(t,x)=0,\quad (t,x)\in Q_{\gamma,\tau,\delta} \cup([0,T]\times V_1),\ee
\bel{t1dd}d^1(t,x)=d^2(t,x),\quad (t,x)\in[0,T]\times (\R^n\setminus B_{R_0})\ee
are fulfilled. Then, for any $\epsilon\in (0,T)$ arbitrary small, the condition 
\bel{t1e} \Lambda_{a^1,B^1,c^1,d^1,\epsilon}=\Lambda_{a^2,B^2,c^2,d^2,\epsilon}\ee
implies that $a^1=a^2$, $B^1=B^2$, $c^1=c^2$, $d^1=d^2$.
\end{thm}

To the best of our knowledge, Theorem \ref{t1} provides the first positive resolution of problem \textbf{(IP)} under the generic assumptions \eqref{t1b}-\eqref{t1dd}. In fact, this theorem establishes what appears to be the most comprehensive coefficient recovery result currently available for classes of both linear and nonlinear hyperbolic equations. Specifically, we relax the assumptions previously imposed on the leading-order coefficients $(a,c)$ in the nonlinear settings of \cite{CLOP,FLO,HUZ,KLU,KLOU}, while simultaneously determining the lower-order coefficients $(B,d)$, which had only been treated separately in \cite{FIKO,FeK,Ki1,KiOk,SY}. Furthermore, in contrast to the existing literature, our approach circumvents the classical obstructions inherent to coefficient identification in hyperbolic equations, such as gauge equivalences and domain-of-dependence restrictions stemming from finite propagation speed. This yields a stronger uniqueness result under minimal hypotheses. The scope of Theorem \ref{t1} highlights the decisive role of nonlocal attenuation in inverse problems for hyperbolic equations, in line with phenomena documented in control theory, where dissipative  effects substantially modify fundamental dynamics such as long-time behavior.

Beyond the broad class of coefficients recovered, Theorem \ref{t1} also constitutes the most general uniqueness result based on data supported on disjoint sets. In particular, it provides the first recovery result for hyperbolic equations from observations supported on domains disjoint in both space and time, thereby significantly extending previous works in this direction \cite{FLO,KKLO,KRU,LO}. We emphasize that formulating inverse problems with data supported on disjoint sets is motivated not only by theoretical interest, but also by experimental setups where the excitation and observation regions cannot physically overlap.

The proof of Theorem \ref{t1} hinges on the interplay between the local hyperbolic dynamics and the specific structure of the nonlocal attenuation terms. More precisely, we exploit the interaction between the local differential operator and the nonlocal components in \eqref{eq1}. Properties of the local wave operator are coupled with memory effects governed by the time-fractional derivative, as well as the strong unique continuation property inherent to the fractional Laplacian. Theorem \ref{t1} thereby illustrates how nonlocal attenuation can serve as an effective mechanism to overcome longstanding bottlenecks in inverse coefficient problems for hyperbolic equations.

\section{Preliminary properties}\label{s3}

The main goal of this section is to establish the well-posedness of problem~\eqref{eq1} and its formal adjoint problem, together with suitable properties of the nonlocal operators appearing in \eqref{eq1}. Namely, we introduce some properties of unique continuation for fractional Laplacian  and memory effect for time fractional derivative as well as other qualitative properties of these operators that will be used in 
the proof of Theorem~\ref{t1}.

\subsection{Forward problem}
We state the uniqueness and existence of solutions to the IVP \eqref{eq1}.

\begin{prop}
\label{p1} Let $c\in C^1([0,T]\times\R^n;(0,+\infty))$, $a=(a_{ij})_{1\leq i,j\leq n}\in C^1([0,T]\times\R^n;\R^{n\times n})$, $B=(b_k)_{1\leq k\leq n+1}\in C^1([0,T]\times\R^n;\R^{n+1})$, $d,e\in C^1([0,T]\times\R^n)$, $F\in L^2((0,T)\times\R^n)$ and assume that 
\eqref{coe1}-\eqref{coe4} are fulfilled. Then there exists a unique solution $u\in C([0,T]; H^1(\R^n))\cap C^1([0,T]; L^2(\R^n))$ to the IVP \eqref{eq1} which satisfies the following estimate 
$$\norm{u}_{C([0,T]; H^1(\R^n))}+\norm{u}_{C^1([0,T]; L^2(\R^n))}\leq C\norm{F}_{L^2((0,T)\times\R^n)},$$
with $C>0$ depending only on $T$, $n$, $\alpha$, $s$, $a$, $B$, $c$, $d$, $e$.

\end{prop}
\begin{proof}
We consider a proof based on a fixed-point argument and a classical unique 
existence result for hyperbolic equations. 
In view of \cite[Chap. 3, Theorem 8.2]{LM1}, for every $\tau\in[0,T]$,  $(v_0,v_1) \in \mathcal H:=H^1(\R^n)\times L^2(\R^n)$, 
we can consider the $ C([0,T];H^1(\R^n)) \cap  C^1([0,T];L^2(\R^n))$-solution 
$v_\tau$ to the problem
\begin{equation}
\label{eq3}
\left\{
\begin{aligned}
& c(t,x)^{-2}\partial_t^2 v_\tau-\sum_{i,j=1}^n\partial_{x_i}(a_{ij}(t,x)\partial_{x_j}v_\tau)+B'(t,x)\cdot\nabla_xv_\tau   = 0, &&\quad \textrm{in } (0,T)\times \R^n,\\
& v_\tau(\tau,x) = v_0(x),\quad \partial_t v_\tau(\tau,x) = v_1(x),                                             &&\quad x\in\R^n,
\end{aligned}
\right.
\end{equation}
where $B=(b_0,B')$.
Then, we introduce the operator 
\begin{equation}
\label{def-U0}
U_0(t,\tau) : \mathcal H\ni(v_0,v_1) \to (v_\tau(t),\partial_t v_\tau(t))\in\mathcal H,\quad t\in[0,T]
\end{equation} 
and recall that $(t,\tau)\mapsto U(t,\tau)\in C([0,T]^2;\mathcal B (\mathcal H))$. 
Here and henceforth, we denote by $\mathcal B(X,Y)$ the set of linear bounded operators from 
the Banach space $X$ to the Banach space $Y$, 
and we write $\mathcal B(X)$ instead of $\mathcal B(X,X)$. 

In light of \cite[Chap. 3, Theorem 8.2]{LM1} and the Duhamel's principle (see e.g. \cite{Ki101,Ki102} and \cite[Section 5]{Pe}), for $d=e\equiv0$,  it is well known that \eqref{eq1} admits 
a unique solution $u\in C^1([0,T];L^2(\R^n))\cap C([0,T];H^1(\R^n))$ such that 
$U_1:=(u,\partial_t u) \in  C([0,T];\mathcal H)$ reads
$$U_1(t)=\int_0^tU_0(t,s)(0,F(s))^Tds,\ t \in [0,T],$$
and satisfies the estimate
\begin{equation}
\label{es}
\norm{U_1}_{ C([0,T];\mathcal H)}
\leq \norm{U_0}_{ C([0,T]^2;\mathcal B(\mathcal H))}\sqrt{T} \norm{F}_{L^2((0,T)\times\R^n)}.
\end{equation}
Let us also fix
$$A(t):=\left(\begin{array}{ll}0&0\\ e(t,\cdot)(-\Delta)^{s}& b_0(t,\cdot)  \end{array}\right)\quad 
H(t,\tau):=\left(\begin{array}{ll}0&0\\ 0& d(t,\cdot)\frac{\tau^{-\alpha}}{\Gamma(1-\alpha)}  \end{array}\right),\quad t,\tau\in (0,T].$$
Let $\mathcal G$ be the map defined on $C([0,T];\mathcal H)$ by 
\begin{equation}
\label{FP-eq2}
(\mathcal GK)(t):=U_1(t)-\int_0^tU_0(t,s)A(s)K(s)ds-\int_0^t\int_0^sU_0(t,s)H(s,s-\tau)K(\tau) d\tau ds,
\end{equation}
for all $K\in C([0,T];\mathcal H)$ and $t\in[0,T]$. It is clear that $\mathcal G$ is a continuous map from $C([0,T];\mathcal H)$ into itself. Moreover, applying Duhamel's principle, we know that
$u \in C([0,T];H^1(\R^n)) \cap C^1([0,T];L^2(\R^n))$ solves
\eqref{eq1} if and only if $U:=(u,\partial_t u)$ is a $C([0,T];\mathcal H)$-solution to 
the integral equation $U=\mathcal G U$. Therefore, we are left with the task of proving that some iterate $\mathcal G^m$ of $\mathcal G$ is a contraction mapping 
on $C([0,T];\mathcal H)$. Recalling that $s\in(0,1/2]$, we observe that, for all $t,\tau \in (0,T]$, we have $A(t),H(t,\tau) \in \mathcal B(\mathcal H)$, and there exists a constant $C>0$ depending on $B,d,e$, $s$, $\alpha$, and $T$ such that 
\begin{equation}
\label{FP-eq3}
\norm{A(t)}_{\mathcal B(\mathcal H)}\leq C,\quad \norm{H(t,\tau)}_{\mathcal B(\mathcal H)} \leq C \tau^{-\alpha},\quad t,\tau\in(0,T).
\end{equation}
Fix $\mathcal J$ defined on $C([0,T];\mathcal H)$ by
$$\mathcal J V:=\mathcal G V - U_1,\quad V\in C([0,T];\mathcal H).$$ Applying Fubini's theorem, for all $t\in[0,T]$, we obtain
\begin{equation}\label{eq-tg}\begin{aligned}
    \norm{\mathcal J W(t)}_{\mathcal H}
&\leq  C\int_0^t\left(\norm{W(s)}_\mathcal H + \int_0^s (s-\tau)^{-\alpha} \norm{W(\tau)}_\mathcal H 
 d\tau\right)ds\\
&\leq  C\left( \int_0^t\norm{W(s)}_\mathcal H ds+ \int_0^t \int_\tau^t (s-\tau)^{-\alpha} \norm{W(\tau)}_\mathcal H 
ds d\tau\right) \nonumber\\
& \leq  C \int_0^t\left(1+\frac{(t-\tau)^{1-\alpha}}{1-\alpha}\right)\norm{W(\tau)}_\mathcal H d\tau\\
&\leq C\left(1+\frac{T^{1-\alpha}}{1-\alpha}\right)\int_0^t\norm{W(\tau)}_\mathcal H d\tau,\quad W\in  C([0,T];\mathcal H).
\end{aligned}
\end{equation}
Combining this estimate with \eqref{es}-\eqref{FP-eq2}, we observe that $\mathcal G$ maps 
$ C([0,T]; \mathcal H)$ into itself and  there exists $m_0\in\mathbb N$ such that $\mathcal G^{m_0}$ is contractive on $ C([0,T];\mathcal H)$. Therefore, from the Banach fixed-point theorem, we deduce that $\mathcal G$ admits a unique fixed point $U\in  C([0,T];\mathcal H)$.
  This proves the unique existence of  a $ C([0,T];H^1(\R^n)) \cap  C^1([0,T];L^2(\R^n))$-solution to 
\eqref{eq1}. In addition, fixing
$$D(t)=\sup_{s\in[0,t]}\norm{U(s)}_\mathcal H,\quad t\in[0,T]$$ and applying \eqref{es}, \eqref{FP-eq2}, \eqref{FP-eq3}, we obtain
$$D(t)\leq C_0\norm{F}_{L^2((0,T)\times\R^n)}+C_1\int_0^tD(s)ds,$$
where $C_0,C_1>0$ depend only on $A$, $B$, $c$, $d$, $e$, $T$, $\alpha$, $s$. Then, estimate \eqref{esta} follows from a direct application of the Gronwall's inequality. This completes the proof of the  lemma.

\end{proof}

We similarly consider the formal adjoint problem to \eqref{eq1} defined by
\begin{equation}
\label{eq1*}
\left\{
\begin{aligned}
& P_{a,B,c}^*v  -\partial_t^{\alpha*} (dv) +(-\Delta)^{s}(ev)  = G, &&\quad \textrm{in } (0,T)\times \R^n,\\                                                    
& v(T,x) = 0,\quad \partial_t v(T,x) = 0,                                             &&\quad x\in\R^n,
\end{aligned}
\right.
\end{equation}
where 
\bel{P*}P_{a,B,c}^*w=\partial_t^2 (c^{-2} w)-\sum_{i,j=1}^n\partial_{x_i}(a_{ij}(t,x)\partial_{x_j}w)-(\partial_t ,\nabla_x)\cdot (Bw),\quad w\in H^1((0,T)\times\R^n),\ee
$$\pa_t^{\alpha*}(d_1 v)(t,x)=\int_t^T\frac{(s-t)^{-\alpha}}{\Gamma(1-\alpha)}\partial_s (d_1v)(s,x)ds,\quad (t,x)\in(0,T)\times \R^n.$$
For any $\gamma\in(0,1)$, we fix also
\bel{I}I_t^\gamma w(t,x)=\int_0^t\frac{(t-s)^{\gamma-1}}{\Gamma(\gamma)}w(s,x)ds,\quad w\in L^1((0,T)\times U_1),\ (t,x)\in(0,T)\times U_1.\ee
We recall the following properties of the fractional derivative in times $\pa_t^{\alpha}$ and its formal adjoint $\pa_t^{\alpha*}$.

\begin{lem}\label{l1} Let $H$ be a Hilbert space and consider $\phi_1,\phi_2\in C^1([0,T];H)$ such that $\phi_1(0)=\phi_2(T)=0$. Then, fixing $\widetilde{\phi_2}(t)=\phi_2(T-t)$, $t\in[0,T]$, the following properties
\bel{l1a}\partial_t^{\alpha*}\phi_2(t)=-\partial_t^\alpha\widetilde{\phi_2}(T-t),\quad t\in[0,T],\ee
\bel{l1b}\pa_t^{\alpha*}\phi_2(t)=\partial_t\left(\int_t^{T}\frac{(\tau-t)^{-\alpha}}{\Gamma(1-\alpha)} \phi_2(\tau)d\tau\right),\quad t\in[0,T],\ee
\bel{l1c}\int_0^T \left\langle \partial_t^{\alpha}\phi_1(t),\phi_2(t)\right\rangle_H dt=-\int_0^T\left\langle \phi_1(t),\partial_t^{\alpha*}\phi_2(t)\right\rangle_H dt,\ee
\bel{l1d}\partial_t^{\alpha}I_t^\alpha\phi_1(t)=\phi_1(t),\quad t\in[0,T]\ee
 hold true.   
\end{lem}
One can refer to \cite[Lemma 3.2]{KRU} 
and \cite[Section 2.3.3]{Po} for the proof of Lemma \ref{l1}.

Similarly to \eqref{eq1}, we can prove the well-posedness of \eqref{eq1*} as follows.

\begin{prop}
\label{p2} Let $c\in C^2([0,T]\times\R^n;(0,+\infty))$, $a=(a_{ij})_{1\leq i,j\leq n}\in C^1([0,T]\times\R^n;\R^{n\times n})$, $B=(b_k)_{1\leq k\leq n+1}\in C^1([0,T]\times\R^n;\R^{n+1})$, $d,e\in C^1([0,T]\times\R^n)$, $G\in L^2((0,T)\times\R^n)$ and assume that 
\eqref{coe1}-\eqref{coe4} are fulfilled. Then there exists a unique solution $v\in C([0,T]; H^1(\R^n))\cap C^1([0,T]; L^2(\R^n))$ to \eqref{eq1*} which satisfies the following estimate 
\bel{esta}\norm{v}_{C([0,T]; H^1(\R^n))}+\norm{v}_{C^1([0,T]; L^2(\R^n))}\leq C\norm{G}_{L^2((0,T)\times\R^n)},\ee
with $C>0$ depending only on $T$, $n$, $\alpha$, $s$, $a$, $B$, $c$, $d$, $e$.

\end{prop}

\begin{proof} From now on for $f\in C([0,T];X)$, with $X$ a Banach space, we denote its time reversal $\tilde{f}(t):=f(T-t)$. Applying Lemma \ref{l1}, one can check that $w=\tilde{v}$ solves the IVP
\begin{equation}
\label{eq2*}
\left\{
\begin{aligned}
& \mathcal Q w+\partial_t^{\alpha} (\tilde{d}w) +(-\Delta)^{s}(\tilde{e}w) = \tilde{G}, &&\quad \textrm{in } (0,T)\times \R^n,\\                                                    
& w(0,x) = 0,\quad \partial_t w(0,x) = 0,                                             &&\quad x\in\R^n,
\end{aligned}
\right.
\end{equation}
 where 
 $$\mathcal Q w=\tilde{c}^{-2}\partial_t^2 w-\sum_{i,j=1}^n\partial_{x_i}(\tilde{a}_{ij}(t,x)\partial_{x_j}w)+B_1(t,x)\cdot\nabla_xw+p(t,x)\partial_tw+q(t,x)w,$$
 with
 $$B_1=-\tilde{B'},\quad p=\tilde{b_0}+2\partial_t(\tilde{c}^{-2}),\quad q=-(-\partial_t,\nabla_x)\cdot \tilde{B}+\partial_t^2(\tilde{c}^{-2}).$$

In view of \cite[Chap. 3, Theorem 8.2]{LM1}, for every $\tau\in[0,T]$,  $(f_0,f_1) \in \mathcal H$, 
we can consider the $ C([0,T];H^1(\R^n)) \cap  C^1([0,T];L^2(\R^n))$-solution 
$y_\tau$ to the problem
\begin{equation}
\label{eq3*}
\left\{
\begin{aligned}
& \tilde{c}^{-2}\partial_t^2 y_\tau-\sum_{i,j=1}^n\partial_{x_i}(\tilde{a}_{ij}(t,x)\partial_{x_j}y_\tau)+B_1(t,x)\cdot\nabla_xy_\tau  +q(t,x) y_\tau= 0, &&\quad \textrm{in } (0,T)\times \R^n,\\
& y_\tau(\tau,x) = f_0(x),\quad \partial_t y_\tau(\tau,x) = f_1(x),                                             &&\quad x\in\R^n.
\end{aligned}
\right.
\end{equation}
Then, we introduce the operator 
$$V_0(t,\tau) : \mathcal H\ni(f_0,f_1) \to (y_\tau(t),\partial_t y_\tau(t))\in\mathcal H$$
and recall that $(t,\tau)\mapsto V_0(t,\tau)\in C([0,T]^2;\mathcal B (\mathcal H))$. 
In light of \cite[Chap. 3, Theorem 8.2]{LM1} and the Duhamel's principle, for $d=e\equiv0$,  it is well known that \eqref{eq2*} admits 
a unique solution $w\in C^1([0,T];L^2(\R^n))\cap C([0,T];H^1(\R^n))$ such that 
$W_1:=(w,\partial_t w) \in  C([0,T];\mathcal H)$ reads
$$W_1(t)=\int_0^tV_0(t,s)(0,\tilde{G}(s))^Tds,\ t \in [0,T].$$
Let us also fix
\begin{equation}
\label{def-Q}A_*(t):=\left(\begin{array}{ll}0&0\\ (-\Delta)^{s}(\tilde{e}(t,\cdot)\cdot)& p(t,\cdot)  \end{array}\right)\quad 
H_*(t,\tau):=\left(\begin{array}{ll}0&0\\ 0& \tilde{d}(\tau,\cdot)\frac{t^{-\alpha}}{\Gamma(1-\alpha)}  \end{array}\right),\  t,\tau\in (0,T]
\end{equation}
where $(-\Delta)^{s}(\tilde{e}(t,\cdot)\cdot)$ is defined by
$$(-\Delta)^{s}(\tilde{e}(t,\cdot)\cdot):H^1(\R^n)\ni v\mapsto (-\Delta)^{s}(\tilde{e}(t,\cdot)v)\in L^2(\R^n).$$
Let $\mathcal G_*$ be the map defined on $C([0,T];\mathcal H)$ by  
$$(\mathcal G_*K)(t):=W_1(t)-\int_0^tV_0(t,s)A_*(s)K(s)ds-\int_0^t\int_0^sV_0(t,s)H_*(s-\tau,\tau)K(\tau) d\tau ds,$$
for all $K\in C([0,T];\mathcal H)$ and $t\in[0,T]$. Applying the Duhamel's principle, we know that
$w \in  C([0,T];H^1(\R^n)) \cap  C^1([0,T];L^2(\R^n))$ solves
\eqref{eq2*} if and only if $W:=(w,\partial_t w)$ is a $ C([0,T];\mathcal H)$-solution to 
the integral equation $W=\mathcal G_*W$.
Repeating the argumentation of Proposition \ref{p1}, we can prove that the exists $m_0\in\mathbb N$ such that $\mathcal G_*^{m_0}$ is contractive on $ C([0,T];\mathcal H)$ and we can complete the proof of the proposition.
 
\end{proof}

\subsection{Unique continuation with moving domain and memory effect}

We first consider a generalization of the unique continuation principle for the fractional Laplacian established in \cite[Theorem 1.2]{GSU} (see also \cite{Ri} for earlier related contributions, as well as \cite{FKU, FLi} for related extensions of this result) to moving domain.

\begin{prop}
\label{p3} Let $R\in C([0,T];L^2(\R^n))$, $0\leq t_1<t_2\leq T$ and $\{\mathcal U(t):\ t\in[t_1,t_2]\}$ a family of not-empty open subset of $\R^n$. Then, the condition
\bel{p3a} R(t,x)=(-\Delta)^{s}R(t,x)=0,\quad t\in(t_1,t_2),\ x\in \mathcal U(t)\ee
implies that $R=0$ on $(t_1,t_2)\times\R^n$.

\end{prop}
\begin{proof}
Fix $t\in(t_1,t_2)$,  $v=R(t,\cdot)\in L^2(\R^n)$,  and $\mathcal O=\mathcal U(t)$. Then, condition \eqref{p3a} implies that 
$$v|_\mathcal O\equiv0,\quad (-\Delta)^{s}v|_\mathcal O\equiv0.$$
Applying \cite[Theorem 1.2]{GSU}, we deduce that $R(t,\cdot)=v= 0$ on $\R^n$. This clearly implies that $R=0$ on $(t_1,t_2)\times\R^n$.

\end{proof}

Let us  recall the definition of the Riemann-Liouville time fractional derivative $D^\alpha_t$ of order $\alpha$
defined on $L^1(0,T; L^2(U_1))$, with $U_1$ an open not empty subset of $\R^n$, by
$$D^\alpha_tw(t,x)=\partial_tI^{1-\alpha}_tw(t,x),\quad w\in L^1(0,T; L^2(U_1)),\ (t,x)\in(0,T)\times U_1.$$
Here the time fractional integral operator $I^{1-\alpha}_t$ is defined by \eqref{I}
 with $\gamma=1-\alpha$. Recalling that, for all $w\in L^1(0,T; L^2(U_1))$,  $I^{1-\alpha}_tw\in L^1(0,T; L^2(U_1))$, we deduce that $D^\alpha_tw$ can be seen as an element of $D'(0,T;L^2(U_1))$. Note also that we have
\bel{RC}D^\alpha_t[w-w(0,\cdot)](t,x)=\partial_t^\alpha w(t,x),\quad w\in W^{1,1}(0,T; L^2(U_1)),\ (t,x)\in(0,T)\times U_1.\ee
We recall the following result of \cite{KRU} (see also \cite{JK} for similar results under stronger regularity assumptions), which highlights the memory effect associated with the time-fractional derivative.

\begin{prop}\label{p4} \emph{(Theorem A.1, \cite{KRU})}
Let $w\in L^1(0,T;L^2(U_1))$, $g\in L^2(U_1)$, and fix $\epsilon\in(0,T)$. Then the condition
\bel{p4a}
w(t,x)=D_t^\alpha (w-g(x))(t,x)=0,\ (t,x)\in (T-\epsilon,T)\times U_1\ee
implies that 
\bel{p4b}w(t,x)=g(x)=0,\ (t,x)\in (0,T)\times U_1.\ee

\end{prop}

\section{Proof of Theorem \ref{t1}}\label{s4}
This section is devoted to the proof of the main result of this article stated in Theorem \ref{t1}. Without loss of generality, we assume that \eqref{t1e} is fulfilled for $\epsilon\in(0,\delta)$ and we will prove that this condition implies that $a^1=a^2$, $B^1=B^2$, $c^1=c^2$, $d^1=d^2$. We divide the proof into 4 steps

\textbf{Step 1.} Applying Proposition \ref{p1},  we can consider $u_j\in C([0,T]; H^1(\R^n))\cap C^1([0,T]; L^2(\R^n))$, $j=1,2$, the solution of  \eqref{eq1} with $a=a^j$, $c=c^j$, $B=B^j$, $d=d^j$ and $F\in C^\infty_0((0,\epsilon)\times\mathcal O_1)$. In this step, we will prove that, for all $F\in C^\infty_0((0,\epsilon)\times\mathcal O_1)$,  $u_1=u_2$ on $[0,T]\times\R^n$. Fix $F\in C^\infty_0((0,\epsilon)\times\mathcal O_1)$, let $u=u_1-u_2\in C([0,T]; H^1(\R^n))\cap C^1([0,T]; L^2(\R^n))$ and observe that $u$ solves the problem
\begin{equation}
\label{eq5}
\left\{
\begin{aligned}
& P_{a^1,B^1,c^1}u +d^1(t,x) \partial_t^{\alpha} u +e(t,x)(-\Delta)^{s}u = \mathcal L u_2, &&\quad \textrm{in } (0,T)\times \R^n,\\                                                    
& u(0,x) = 0,\quad \partial_t u(0,x) = 0,                                             &&\quad x\in\R^n,
\end{aligned}
\right.
\end{equation}
where 
\bel{L}\begin{aligned}\mathcal L u_2=&((c^2)^{-2}-(c^1)^{-2})\partial_t^2 u_2-\sum_{i,j=1}^n\partial_{x_i}((a_{ij}^2-a_{ij}^1)\partial_{x_j}u_2)+(B^2-B^1)\cdot(\partial_t u_2,\nabla_xu_2)\\
&+(d^2-d^1)\partial_t^\alpha u_2.\end{aligned}\ee
In light of \eqref{coe2} and \eqref{t1dd}, we have
\bel{tata}\mathcal L u_2(t,x)=0,\quad (t,x)\in[0,T]\times (\R^n\setminus \overline{B_{R_0}})\supset Q_{\gamma,\tau,\epsilon}.\ee
Moreover, from condition \eqref{t1e}, we find $u|_{(T-\epsilon,T)\times \mathcal O_2}\equiv 0$ which implies that
$$P_{a^1,B^1,c^1}u(t,x)=0,\quad (t,x)\in (T-\epsilon,T)\times \mathcal O_2\supset Q_{\gamma,\tau,\epsilon}.$$
Combining these properties with \eqref{t1d}, we obtain
$$d^1(t,x)\partial_t^\alpha u(t,x)=\mathcal L u_2(t,x)=P_{a^1,B^1,c^1}u(t,x)=0,\quad (t,x)\in Q_{\gamma,\tau,\epsilon}.$$
Thus, \eqref{eq5} implies
$$\begin{aligned}&e(t,x)(-\Delta)^{s}u(t,x)=-P_{a^1,B^1,c^1}u(t,x)-d^1(t,x)\partial_t^\alpha u(t,x)+\mathcal L u_2(t,x)  =0,\quad  (t,x)\in Q_{\gamma,\tau,\epsilon}\end{aligned}$$
and, conditions \eqref{t1b} and \eqref{t1e} imply that
$$u(t,x)=(-\Delta)^{s}u(t,x)=0,\quad  (t,x)\in Q_{\gamma,\tau,\epsilon}.$$
Thus, applying Proposition \ref{p3} with  $t_1=T-\epsilon$, $t_2=T$, $\mathcal U(t)=B(\gamma(t),\tau(t))$, $t\in[t_1,t_2]$, $R=u$, we deduce that $u=0$ on $(T-\epsilon,T)\times\R^n$. Thus, in light of \eqref{tata}, we have
\bel{t1h}d^1 \partial_t^{\alpha} u(t,x)=(\mathcal L u_2-P_{a^1,B^1,c^1}u  -e_1(-\Delta)^{s}u)(t,x)=0,\quad (t,x)\in(T-\epsilon,T)\times(\R^n\setminus \overline{B_{R_0}}).\ee
On the other hand, in view of \eqref{t1c}, there exists $\epsilon_1\in(0,\epsilon)$ and $U_2\subset U$ an open not empty subset of $\R^n$, such that
$$\inf_{(t,x)\in [T-\epsilon_1,T]\times \overline{U_2}} |d^1(t,x)|>0.$$
Combining this with \eqref{U}, \eqref{t1h} and the fact that $u=0$ on $(T-\epsilon,T)\times\R^n$, we obtain
$$\partial_t^{\alpha} u(t,x)=u(t,x)=0,\quad (t,x)\in(T-\epsilon_1,T)\times U_2.$$
Therefore, applying \eqref{RC} and Proposition \ref{p4}, we find
$$u(t,x)=0,\quad (t,x)\in (0,T)\times U_2.$$
In light of \eqref{U} and \eqref{tata}, it follows that
$$\mathcal L u_2(t,x)=d^1(t,x)\partial_t^\alpha u(t,x)=P_{a^1,B^1,c^1}u(t,x)=0,\quad (t,x)\in (0,T)\times U_2$$
and we deduce from \eqref{t1b} that
$$\begin{aligned}(-\Delta)^{s}u(t,x)=\frac{\mathcal L u_2(t,x)-P_{a^1,B^1,c^1}u(t,x)-d^1(t,x)\partial_t^\alpha u(t,x)}{e(t,x)}=0,\quad (t,x)\in (0,T)\times U_2.\end{aligned}$$
Thus, we have
$$u(t,x)=(-\Delta)^{s}u(t,x)=0,\quad  (t,x)\in (0,T)\times U_2$$
and, applying again Proposition \ref{p3}, we get
$$u(t,x)=0,\quad (t,x)\in(0,T)\times\R^n.$$
This proves that  $u_1=u_2$ on $[0,T]\times\R^n$ for any arbitrary $F\in C^\infty_0((0,\epsilon)\times\mathcal O_1)$.\\

\textbf{Step 2.} Consider $\mathcal L^*$ defined, by 
$$\begin{aligned}\mathcal L^*\phi=&(\partial_t^2(((c^2)^{-2}-(c^1)^{-2})\phi)-\sum_{i,j=1}^n\partial_{x_j}((a_{ij}^2-a_{ij}^1)\partial_{x_i}\phi)-(\partial_t ,\nabla_x)\cdot((B^2-B^1)\phi))\\
&-\partial_t^{\alpha*}[(d^2-d^1)\phi],\quad \phi\in C^\infty_0((0,T)\times\R^n).\end{aligned}$$
In this step, we will show that
\bel{t1n}\mathcal L^*\psi(t,x)=0,\quad \psi\in  C^\infty_0((0,T)\times\R^n;\mathbb C),\ (t,x)\in [0,T]\times\R^n.\ee
Combining the result of Step 1 with \eqref{eq5}, we obtain
\bel{t1i}\mathcal L u_2(t,x)=(P_{a^1,B^1,c^1}u +d_1 \partial_t^{\alpha} u +e(-\Delta)^{s}u)(t,x) =0,\quad (t,x)\in (0,T)\times\R^n.\ee
Let $\psi\in C^\infty_0((0,T)\times\R^n;\mathbb C)$ and $F\in C^\infty_0((0,\epsilon)\times\mathcal O_1)$.  Using \eqref{L}, \eqref{t1i}, Lemma \ref{l1} and integrating by parts, we get
\bel{t1j}\begin{aligned}0&=\left\langle \mathcal L u_2, \overline{\psi}\right\rangle_{D'((0,T)\times\R^n),C^\infty_0((0,T)\times\R^n)}\\
&=\int_0^T\int_{\R^n} (-\partial_t u_2\partial_t(((c^2)^{-2}-(c^1)^{-2})\overline{\psi})+\sum_{i,j=1}^n((a_{ij}^2-a_{ij}^1)\partial_{x_j}u_2)\partial_{x_i}\overline{\psi})dxdt\\
&\ \ \ +\int_0^T\int_{\R^n}(d^2-d^1)\partial_t^\alpha u_2\overline{\psi}-u_2(\partial_t ,\nabla_x)\cdot((B^2-B^1)\overline{\psi}) dxdt\\
&=\int_0^T\int_{\R^n} u_2(\partial_t^2(((c^2)^{-2}-(c^1)^{-2})\overline{\psi})-\sum_{i,j=1}^n\partial_{x_j}((a_{ij}^2-a_{ij}^1)\partial_{x_i}\overline{\psi}-(\partial_t ,\nabla_x)\cdot((B^2-B^1)\overline{\psi}))dxdt\\
&\ \ \ -\int_0^T\int_{\R^n}u_2\partial_t^{\alpha*}[(d^2-d^1)\overline{\psi}] dxdt\\
&=\int_0^T\int_{\R^n} u_2\overline{\mathcal L^*\psi} dxdt.\end{aligned}\ee
Now let us consider the following IVP  
\begin{equation}
\label{eq5*}
\left\{
\begin{aligned}
& P_{a^2,B^2,c^2}^*v  -\partial_t^{\alpha*} (d^2v) +(-\Delta)^{s}(ev) =\mathcal L^*\psi, &&\quad \textrm{in } (0,T)\times \R^n,\\                                                    
& v(T,x) = 0,\quad \partial_t v(T,x) = 0,                                             &&\quad x\in\R^n.
\end{aligned}
\right.
\end{equation}
In light of Proposition \ref{p2} the IVP \eqref{eq5*} admits a unique solution $v\in C([0,T]; H^1(\R^n))\cap C^1([0,T]; L^2(\R^n))$. Since $v$ solves \eqref{eq5*} and $u_2$ solves \eqref{eq1} with $(a,B,c,d)=(a^2,B^2,c^2,d^2)$, it is clear that $P_{a^2,B^2,c^2}^*v\in C([0,T];H^{-1}(\R^n))$ and $P_{a^2,B^2,c^2}u_2\in C([0,T];H^{-1}(\R^n))$. Moreover, integrating by parts, we get
$$\left\langle  P_{a^2,B^2,c^2}^*v, u_2\right\rangle_{L^2(0,T;H^{-1}(\R^n)),L^2(0,T;H^{1}(\R^n))}=\left\langle  P_{a^2,B^2,c^2}u_2, v\right\rangle_{L^2(0,T;H^{-1}(\R^n)),L^2(0,T;H^{1}(\R^n))}.$$
In the same way, we find
$$\int_0^T\int_{\R^n} u_2(\overline{(-\Delta)^{s}(ev)})dxdt=\int_0^T\int_{\R^n} (e (-\Delta)^{s}u_2 )\overline{v}dxdt
$$
and, applying Lemma \ref{l1}, we find
$$\int_0^T\int_{\R^n} u_2\overline{\partial_t^{\alpha*} (d^2v)}dxdt=-\int_0^T\int_{\R^n} (d^2\partial_t^\alpha u_2) \overline{v}dxdt.
$$
Combining these formula with \eqref{t1j}, we get
$$\begin{aligned}0&=\int_0^T\int_{\R^n} u_2 \overline{\mathcal L^*\psi} dxdt\\
&=\overline{\left\langle  P_{a^2,B^2,c^2}^*v  -\partial_t^{\alpha*} (d^2v) +(-\Delta)^{s}(ev), u_2\right\rangle}_{L^2(0,T;H^{-1}(\R^n)),L^2(0,T;H^{1}(\R^n))}\\
&=\left\langle P_{a^2,B^2,c^2}u_2 +d^2 \partial_t^{\alpha} u_2 +e(-\Delta)^{s}u_2, v\right\rangle_{L^2(0,T;H^{-1}(\R^n)),L^2(0,T;H^{1}(\R^n))}\\
&=\int_0^T\int_{\R^n}F\overline{v}dxdt.
\end{aligned}$$
Recalling that in this identity $F\in C^\infty_0((0,\epsilon)\times\mathcal O_1)$ is arbitrary chosen, we deduce that
\bel{t1k} v(t,x)=0,\quad (t,x)\in (0,\epsilon)\times\mathcal O_1.  \ee
Combining this with \eqref{U} and \eqref{t1d}-\eqref{t1dd}, we obtain
$$P_{a^2,B^2,c^2}^*v(t,x)=\partial_t^{\alpha*} (d^2v)(t,x)=\partial_t^{\alpha*} (d^1v)(t,x)=0,\quad (t,x)\in (0,\epsilon)\times V_1.$$
Moreover, conditions   \eqref{coe2},  \eqref{U} and \eqref{t1dd} imply that
\bel{tutu}\mathcal L^*\psi(t,x)=0,\quad (t,x)\in[0,T]\times (\R^n\setminus \overline{B_{R_0}})\supset [0,T]\times \mathcal O_1 \supset [0,T]\times V_1\ee
and it follows
$$\begin{aligned}(-\Delta)^{s}(ev)(t,x)=\mathcal L^*\psi(t,x)-P_{a^2,B^2,c^2}^*v(t,x)  -\partial_t^{\alpha*} (d^2v) (t,x)=0,\quad (t,x)\in(0,\epsilon)\times V_1.\end{aligned}$$
Thus, fixing $R=ev$,  $t_1=0$, $t_2=\epsilon$, $\mathcal U(t)=V_1$, $t\in[t_1,t_2]$,  we deduce that \eqref{p3a} holds true. Then, Proposition \ref{p3} implies that
$$e(t,x)v(t,x)=0,\quad (t,x)\in (0,\epsilon)\times\R^n.$$
Then, using \eqref{t1k} and \eqref{tutu}, we find
$$\partial_t^{\alpha*} (d^2v)(t,x)=(P_{a^2,B^2,c^2}^*v   +(-\Delta)^{s}(ev)-\mathcal L^*\psi)(t,x)=0,\quad (t,x)\in(0,\epsilon)\times\mathcal O_1$$
and, conditions \eqref{U} and \eqref{t1k} imply that
$$d^2v(t,x)=\partial_t^{\alpha*} (d^2v)(t,x)=0,\quad (t,x)\in(0,\epsilon)\times V_2.$$
Fixing $z=\widetilde{d^2v}\in C^1([0,T];L^2(\R^n))$ and applying \eqref{l1a}, we obtain
$$z(t,x)=\partial_t^{\alpha} z(t,x)=0,\quad (t,x)\in(T-\epsilon,T)\times V_2$$
and Proposition \ref{p4} implies that $z=0$ on $(0,T)\times V_2$. Therefore, we have $d^2v=0$ on $(0,T)\times V_2$ and conditions \eqref{t1c}, \eqref{t1dd} imply that $v=0$ on $(0,T)\times V_2$.
Then, in view of \eqref{tutu}, we get
$$\begin{aligned}&(-\Delta)^{s}(ev)(t,x)=\mathcal L^*\psi(t,x) -P_{a^2,B^2,c^2}^*v(t,x)  -\partial_t^{\alpha*} (d^2v)(t,x) =0,\quad (t,x)\in(0,T)\times V_2\end{aligned}$$
and, applying again Proposition \ref{p3}, we find
\bel{t1l}e(t,x)v(t,x)=0,\quad (t,x)\in (0,T)\times\R^n.\ee
Thus, in view of condition \eqref{t1b}, we obtain
\bel{t1m}v(t,x)=0,\quad (t,x)\in (0,T)\times B_{R_0}.\ee
Combining \eqref{t1l}-\eqref{t1m}, we get
$$P_{a^2,B^2,c^2}^*v(t,x)=  \partial_t^{\alpha*} (d^2v)(t,x)=0,\quad (t,x)\in (0,T)\times B_{R_0},$$
$$(-\Delta)^{s}(ev)(t,x)=0,\quad (t,x)\in (0,T)\times\R^n,$$
and we find
$$\mathcal L^*\psi(t,x)=[P_{a^2,B^2,c^2}^*v -\partial_t^{\alpha*} (d^2v) +(-\Delta)^{s}(ev)](t,x)=0,\quad  (t,x)\in (0,T)\times B_{R_0}.$$
Moreover, recalling from \eqref{coe2} and \eqref{t1dd} that supp$(\mathcal L^*\psi)\subset [0,T]\times \overline{B_{R_0}}$, we deduce that 
$\mathcal L^*\psi\equiv0$. Since $\psi\in C^\infty_0((0,T)\times\R^n;\mathbb C)$ is arbitrary chosen, we have proved that \eqref{t1n} holds true.\\

\textbf{Step 3.} From now on we will complete the proof of Theorem \ref{t1} by considering \eqref{t1n} with different choices of the map $\psi\in  C^\infty_0((0,T)\times\R^n;\mathbb C)$.  In this step, we will show that
\bel{t1o}c^1=c^2,\quad a^1=a^2.\ee
For this purpose, fix $\lambda>1$, $(t_0,x_0)\in (0,T)\times \R^n$, $\chi\in C^\infty_0((0,T)\times\R^n)$ such that $\chi=1$ on a neighborhood of $(t_0,x_0)$. Then, choose
$$\psi(t,x)=e^{i\lambda (t-t_0)}\chi(t,x),\quad (t,x)\in [0,T]\times\R^n$$
and deduce from \eqref{t1n} that we have
$$\begin{aligned}0&=\mathcal L^*\psi(t_0,x_0)\\
&=(\partial_t^2(((c^2)^{-2}-(c^1)^{-2})\psi)(t_0,x_0)-\sum_{i,j=1}^n\partial_{x_j}((a_{ij}^2-a_{ij}^1)\partial_{x_i}\psi)(t_0,x_0)\\
&\ \ \ -(\partial_t ,\nabla_x)\cdot((B^2-B^1)\psi))(t_0,x_0)-\partial_t^{\alpha*}[(d^2-d^1)\psi](t_0,x_0).\end{aligned}$$
On the other hand, for $j=1,\ldots,n$, we have $\partial_t\chi=\partial_{x_j}\chi=0$
on a neighborhood of $(t_0,x_0)$, which implies that
\bel{t1p}\begin{aligned}0
&=(\partial_t^2(((c^2)^{-2}-(c^1)^{-2})\psi)(t_0,x_0)-\sum_{i,j=1}^n\partial_{x_j}((a_{ij}^2-a_{ij}^1)\partial_{x_i}\psi(t_0,x_0)\\
&\ \ \ -(\partial_t ,\nabla_x)\cdot((B^2-B^1)\psi))(t_0,x_0)-\partial_t^{\alpha*}[(d^2-d^1)\psi](t_0,x_0)\\
&=-\lambda^2\rho(t_0,x_0)+2i\lambda \partial_t\rho(t_0,x_0)+\partial_t^2\rho(t_0,x_0)-(\partial_t ,\nabla_x)\cdot B(t,x_0)-i\lambda b_0(t,x_0)\\
&\ \ \ -\partial_t^{\alpha*}(d\psi)(t_0,x_0),\end{aligned}\ee
where $\rho=(c^2)^{-2}-(c^1)^{-2}$, $b_0=(b_0^2-b_0^1)$, $B=B^2-B^1$, $d=d^2-d^1$.
Notice that
$$|\partial_t^{\alpha*}(d\psi)(t_0,x_0)|\leq C\norm{d\psi}_{C^1([0,T];L^\infty(\R^n))}\leq C\lambda,$$
with $C>0$ a constant independent of $\lambda$. Thus, dividing both sides of the identity \eqref{t1p} by $\lambda^2$, we obtain
$$\begin{aligned}&|\rho(t_0,x_0)|\\
&\leq \lambda^{-2}\abs{2i\lambda \partial_t\rho(t_0,x_0)+\partial_t^2\rho(t,x_0)-(\partial_t ,\nabla_x)\cdot B(t,x_0)-\partial_tb_0(t,x_0)-i\lambda b_0(t,x_0)}\\
&\ \ \ +\lambda^{-2}\abs{\partial_t^{\alpha*}(d\psi)(t_0,x_0)}\\
&\leq C\lambda^{-1},
    \end{aligned}$$
with $C>0$ independent of $\lambda>1$. Thus, sending $\lambda\to+\infty$, we find 
$$(c^2(t_0,x_0))^{-2}-(c^1(t_0,x_0))^{-2}=\rho(t_0,x_0)=0.$$
Recalling that $(t_0,x_0)\in (0,T)\times \R^n$ is arbitrary chosen, we deduce that $c^1=c^2$.
Similarly, fix $\xi=(\xi_1,\ldots,\xi_n)\in\R^n$ and consider now that
$$\psi(t,x)=e^{i\lambda (x-x_0)\cdot\xi}\chi(t,x),\quad (t,x)\in [0,T]\times\R^n.$$
Then, recalling that $c^1=c^2$ and $\partial_t\chi=\partial_{x_j}\chi=0$, $j=1,\ldots,n$,
on a neighborhood of $(t_0,x_0)$, we get
\bel{t1q}\begin{aligned}0&=\mathcal L^*\psi(t_0,x_0)\\
&=-\sum_{i,j=1}^n\partial_{x_j}((a_{ij}^2-a_{ij}^1)\partial_{x_i}\psi)(t_0,x_0)-(\partial_t ,\nabla_x)\cdot(B\psi))(t_0,x_0) -\partial_t^{\alpha*}[d\psi](t_0,x_0)\\
&=\lambda^2\sum_{i,j=1}^na_{ij}(t_0,x_0)\xi_i\xi_j-i\lambda \sum_{i,j=1}^n\partial_{x_j}a_{ij}(t_0,x_0)\xi_i-(\partial_t ,\nabla_x)\cdot B(t_0,x_0)-i\lambda\sum_{k=1}^nb_k(t_0,x_0)\xi_k\\
&\ \ \ -\partial_t^{\alpha*}[d\chi](t_0,x_0),\end{aligned}\ee
with $a_{ij}=a_{ij}^2-a_{ij}^1$, $i,j=1,\ldots,n$, $(b_k)_{0\leq k\leq n}=B=(b_k^2-b_k^1)_{0\leq k\leq n}$. Again, dividing by $\lambda^2$ and sending $\lambda\to+\infty$, we get
$$\sum_{i,j=1}^na_{ij}(t_0,x_0)\xi_i\xi_j=0.$$
Now recalling that $\xi=(\xi_1,\ldots,\xi_n)\in\R^n$ is arbitrary chosen and recalling, from \eqref{coe3}, that
$$a_{ji}(t_0,x_0)=a_{ji}^2(t_0,x_0)-a_{ji}^1(t_0,x_0)=a_{ij}^2(t_0,x_0)-a_{ij}^1(t_0,x_0)=a_{ij}(t_0,x_0),\quad i,j=1,\ldots,n,$$
by comparison principle, we get
$$a^2(t_0,x_0)-a^1(t_0,x_0)=a(t_0,x_0)=0.$$
Then it follows that $a^1=a^2$ and we obtain \eqref{t1o}.\\

\textbf{Step 4.} In this step we complete the proof of the theorem by showing that $B^1=B^2$ and $d^1=d^2$. Let us first prove that $B^1=B^2$. Combining \eqref{t1o} with \eqref{t1q}, for every $\lambda>1$, $\xi=(\xi_1,\ldots,\xi_n)\in\R^n$ and $(t_0,x_0)\in (0,T)\times \R^n$, we find
$$-(\partial_t ,\nabla_x)\cdot B(t,x_0)-i\lambda\sum_{k=1}^nb_k(t_0,x_0)\xi_k-\partial_t^{\alpha*}[d\chi](t_0,x_0)=0.$$
Dividing by $\lambda$ and sending $\lambda\to+\infty$, we deduce that
$$\sum_{k=1}^n(b_k^2(t_0,x_0)-b_k^1(t_0,x_0))\xi_k=0,\quad (t_0,x_0)\in (0,T)\times \R^n,\ \xi=(\xi_1,\ldots,\xi_n)\in\R^n.$$
It follows that 
\bel{t1r}b^2_k=b^1_k,\quad k=1,\ldots,n.\ee
In the same way, \eqref{t1o} and \eqref{t1p} imply
$$-(\partial_t ,\nabla_x)\cdot B(t,x_0)-i\lambda b_0(t,x_0)-\partial_t^{\alpha*}(de^{i\lambda (t-t_0)}\chi)(t_0,x_0)=0,\quad (t_0,x_0)\in (0,T)\times \R^n.$$
Dividing by $\lambda$ and sending $\lambda\to +\infty$ we get
\bel{t1s}|b_0(t_0,x_0)|=\limsup_{\lambda\to+\infty}\lambda^{-1}|\partial_t^{\alpha*}(de^{i\lambda (t-t_0)}\chi)(t_0,x_0)|.\ee
On the other hand, applying \eqref{coe1}, \eqref{l1a}, fixing $\phi(t)=d(t,x_0)e^{i\lambda (t-t_0)}\chi(t,x_0)$, $t\in[0,T]$,
and choosing $\lambda>(T-t_0)^{-1}+1$, we find
$$\begin{aligned}\abs{\partial_t^{\alpha*}(de^{i\lambda (t-t_0)}\chi)(t_0,x_0)}&=\abs{\partial_t^{\alpha*}\phi(t_0)}\\
&=\abs{\partial_t^{\alpha}\tilde{\phi}(T-t_0)}\\
&=\abs{\int_0^{T-t_0}\frac{(T-t_0-s)^{-\alpha}}{\Gamma(1-\alpha)}\tilde{\phi}'(s)ds}\\
&=\abs{\int_0^{T-t_0}\frac{s^{-\alpha}}{\Gamma(1-\alpha)}\tilde{\phi}'(T-t_0-s)ds}\\
&\leq \abs{\int_0^{\lambda^{-1}}\frac{s^{-\alpha}}{\Gamma(1-\alpha)}\phi'(t_0+s)ds}+\abs{\int_{\lambda^{-1}}^{T-t_0}\frac{s^{-\alpha}}{\Gamma(1-\alpha)}\phi'(t_0+s)ds}\\
&\leq 3\norm{d}_{W^{1,\infty}((0,T)\times B_{R_1})}\norm{\chi}_{W^{1,\infty}((0,T)\times\R^n)} \lambda \int_0^{\lambda^{-1}}s^{-\alpha}ds\\
&\ \ \ +\abs{\frac{(T-t_0)^\alpha \phi(T)}{\Gamma(1-\alpha)}-\frac{\lambda^\alpha \phi(t_0+\lambda^{-1})}{\Gamma(1-\alpha)}+\alpha \int_{\lambda^{-1}}^{T-t_0}\frac{s^{-1-\alpha}}{\Gamma(1-\alpha)}\phi(t_0+s)ds}\\
&\leq C\lambda^{\alpha} +\norm{d}_{L^\infty((0,T)\times B_{R_1})}\norm{\chi}_{L^\infty((0,T)\times\R^n)}\left(\frac{(T-t_0)^\alpha }{\Gamma(1-\alpha)}+\frac{\lambda^\alpha }{\Gamma(1-\alpha)}\right)\\
&\ \ \ +\norm{d}_{L^\infty((0,T)\times B_{R_1})}\norm{\chi}_{L^\infty((0,T)\times\R^n)}\left(\int_{\lambda^{-1}}^{T-t_0}\frac{s^{-1-\alpha}}{\Gamma(1-\alpha)}ds\right)\\
&\leq C\lambda^\alpha,\end{aligned}$$
where $C>0$ is a constant independent of $\lambda>1$ that might change from line to line. Combining this with \eqref{t1s}, we deduce that
$$|b_0(t_0,x_0)|=\limsup_{\lambda\to+\infty}\lambda^{-1}|\partial_t^{\alpha*}(de^{i\lambda (t-t_0)}\chi)(t_0,x_0)|\leq C\limsup_{\lambda\to+\infty}\lambda^{\alpha-1}=0.$$
This proves that $b_0^1=b_0^2$ and, with \eqref{t1r}, it implies that $B^1=B^2$. In order to complete the proof of the theorem, we only need to show that $d^1=d^2$. In light of \eqref{t1n}, \eqref{t1o},  and the fact that $B^1=B^2$, we have
\bel{t1t}\partial_t^{\alpha*}(d\psi)(t,x)=0,\quad \psi\in C^\infty_0((0,T)\times\R^n),\ (t,x)\in [0,T]\times\R^n.\ee
Let us fix $\psi\in C^\infty_0((0,T)\times\R^n)$ and consider $\psi_1\in C^\infty_0((0,T)\times\R^n)$ such that $\psi_1=1$ on a neighborhood of supp$(\psi)$. Multiplying \eqref{t1t} by $I_t^\alpha\psi_1(t,x)$ and integrating on $(0,T)\times\R^n$, we get
$$\int_0^T\int_{\R^n}\partial_t^{\alpha*}(d\psi)(t,x)I_t^\alpha\psi_1(t,x)dxdt=0.$$ Notice also that $I_t^\alpha\psi_1\in C^\infty([0,T];C^\infty_0(\R^n))$ and $$I_t^\alpha\psi_1(0,x)=\psi_1(0,x)=d\psi(T,x)=0,\quad x\in\R^n.$$ Thus, applying \eqref{l1c}-\eqref{l1d}, we find
$$\begin{aligned}0&=\int_0^T\int_{\R^n}\partial_t^{\alpha*}(d\psi)(t,x)I_t^\alpha\psi_1(t,x)dxdt\\
&=-\int_0^T\int_{\R^n}d\psi(t,x)\partial_t^{\alpha}I_t^\alpha\psi_1(t,x)dxdt=-\int_0^T\int_{\R^n}d\psi(t,x)\psi_1(t,x)dxdt\\
&=-\int_0^T\int_{\R^n}(d^2-d^1)\psi dxdt.\end{aligned}$$
Recalling that here $\psi\in C^\infty_0((0,T)\times\R^n)$ is arbitrary chosen this implies that $d^1=d^2$ and it completes the proof of the theorem.
\section*{Acknowledgment}
Y. Kian thanks Ali Feizmohammadi for fruitful discussions as well as its suggestions and comments on this problem. 


\end{document}